\documentclass[12pt]{amsart}
\usepackage{amssymb,latexsym, amscd, a4wide}
\usepackage[all]{xy}
\usepackage{graphicx}
\usepackage{mathrsfs}
\usepackage{amsmath}
\usepackage{pb-diagram}
\usepackage[paper=a4paper,left=30mm,right=20mm,top=25mm,bottom=30mm]{geometry}

 \newtheorem{theorem}{Theorem}

 \newtheorem{cor}{Corollary}

 \newtheorem{lemma}{Lemma}
  \newtheorem{remark}{Remark}

 \newtheorem{proposition}{Proposition} \theoremstyle{definition}
  \newtheorem*{definition-proposition}{Definition-Proposition} \theoremstyle{definition}
  \theoremstyle{definition}
 \theoremstyle{definition}
 \theoremstyle{definition}
 
  \newtheorem*{guiding example}{Guiding example}

 \numberwithin{equation}{section}
 \newcommand{\nc}{\newcommand}
\nc{\ssn}{\subsection{}} \nc{\sssn}{\subsubsection{}}

\numberwithin{equation}{section}

\nc{\oR}{\ol{R}}
\newcommand{\ol}{\ensuremath{\overline}}

\nc{\id}{\textrm{id}}
\newcommand{\ov}{\ensuremath{\overline}}
\newcommand{\oGf}{\ensuremath{\overline{\Gamma}_f }}

\newcommand{\N}{\mathcal{N}}
\newcommand{\hlambda}{\ensuremath{{\widehat{\lambda}}}}
\newcommand{\hmu}{\ensuremath{{\widehat{\mu}}}}

\newcommand{\D}{\Delta}

\newcommand{\Bb}{\mathbb{B}}

\renewcommand{\H}{\mathcal{H}}

\renewcommand{\P}{\mathcal{P}}

\newcommand{\M}{\operatorname{M}}
\newcommand{\Ac}{{\mathcal{A}}}

\newcommand{\hlambdam}{\operatorname{Hom}}

\newcommand{\Ann}{\operatorname{Ann}}
\newcommand{\Hom}{\operatorname{Hom}}

\newcommand{\F}{\ensuremath{{\mathcal{F}}}}

\newcommand{\p}{\ensuremath{\mathfrak{p}}}
\renewcommand{\l}{\ensuremath{\mathfrak{l}}}

\renewcommand{\mod}{\ensuremath{\operatorname{mod}}}

\newcommand{\C}{\ensuremath{\mathbb{C}}}

\newcommand{\isoto}{\ensuremath{\overset{\sim}{\longrightarrow}}}

\newcommand{\Ker}{\operatorname{Ker}}

\newcommand{\ot}{\operatorname{\otimes}}

\newcommand{\g}{\ensuremath{\mathfrak{g}}}
\newcommand{\h}{\ensuremath{\mathfrak{h}}}
\renewcommand{\b}{\ensuremath{\mathfrak{b}}}

\newcommand{\V}{\ensuremath{\mathbb{V}}}

\newcommand{\A}{\ensuremath{{\mathcal{H}}}}

\renewcommand{\O}{\ensuremath{\mathcal{O}}}

\renewcommand{\M}{\mathcal{M}}

\newcommand{\n}{\ensuremath{\mathfrak{n}}}

\begin{document}
\author{Juan Camilo Arias  and Erik Backelin}
\title{Projective and Whittaker functors on category $\O$.}
\address{Erik Backelin, Departamento de Matem\'{a}ticas, Universidad de los Andes,
Carrera 1 N. 18A - 10, Bogot\'a, COLOMBIA}
\email{erbackel@uniandes.edu.co}
\address{Juan Camilo Arias, Departamento de Matem\'{a}ticas, Universidad de los Andes, Carrera 1 N. 18A - 10, Bogot\'a, COLOMBIA. Tel: +571 3394999 ext. 3705 Fax: +571 3324427 }
\email{jc.arias147@uniandes.edu.co }

\keywords{}

\maketitle

\begin{abstract} We show that the Whittaker functor on a regular block of the BGG-category $\O$ of a semisimple complex Lie algebra can be obtained by composing a translation to the wall 
functor with Soergel and Mili\v{c}i\'{c}'s equivalence between the category of Whittaker modules and a singular block of $\O$. We show that the Whittaker functor is a quotient functor that commutes with 
all projective functors and endomorphisms between them.
\end{abstract}

        


\section{Introduction}
Let $\g$ be a complex semisimple Lie algebra, $\b$ a Borel subalgebra, $\h$ a Cartan subalgebra and $\n = [\b,\b]$. Let $\O$ be the Bernstein-Gelfand-Gelfand (BGG) category of representations of $\g$. 
Let $U = U(\g)$ be the universal enveloping algebra of $\g$
and $Z$ its center.
Let $f:U(\n) \to \mathbb{C}$ be an algebra homomorphism and $\N_f$ the corresponding category of Whittaker modules (finitely generated $U$-modules that are locally finite over $Z$ and locally annihilated by $\Ker f$).
For a $U$-module $V$ let $\Gamma_f(V)$ denote the submodule of vectors annihilated by some power of $\Ker f$. In \cite{Bac} the Whittaker functor
$$\oGf: \O \to \N_f, \ M \mapsto \Gamma_f(\ov{M})$$
was introduced. Here $\ov{M}$ denote the completion of $M$, i.e. the product of its weight spaces.
It is exact and maps Verma modules to standard Whittaker modules. In this paper we shall establish a few basic properties of this functor that we haven't found in the literature. 

\medskip

When $f$ is a regular character (i.e. $f(E_\alpha) \neq 0$ for all simple root vectors $E_\alpha$) Kostant \cite{K} showed that
$\N_f$ is equivalent to the category of finite dimensional $Z$-modules. The functor is $M \mapsto Wh_f(M) = \{m \in M | \Ker f \cdot m = 0\}$.

On the other hand, if we fix a dominant weight $\lambda \in \h^*$ and the corresponding block $\O_\hlambda \subset \O$
there is Soergel's functor $\V: \O_\hlambda \to Z-\mod$, $\V(M) = \Hom(P_{w_0 \cdot \lambda}, M)$, where $P_{w_0 \cdot \lambda}$ is the anti-dominant projective. It was shown in \cite{Bac} that if we restrict 
$\oGf$ to $\O_\hlambda$ and compose it with Kostant's equivalence $Wh_f$ the resulting functor is naturally equivalent to $\V$ (see Proposition \ref{aprop} for a short proof). 
Thus one may think of $\oGf$ as a "partial" Soergel functor, for a general character $f$. 

Let $\lambda$ be dominant, integral and regular and $\mu$ be dominant and integral and assume that $f$ is a character such that $W_\mu = W_f$. Then
using the machinery of Harish-Chandra bimodules Soergel and Mili\v{c}i\'{c}, \cite{MS}  established an equivalences between (an enhancement of) $\O_\hmu$ and the block $\N^f_\hlambda$.
Harish-Chandra bimodule theory provides an equivalence $\kappa: \O_\hlambda \isoto \O'_\lambda$. If we compose $\kappa$ with the translation to the wall functor $\Theta^\mu_\lambda :  \O'_\lambda \to \O'_\hmu$ and 
then compose the result with the equivalence $\sigma: \O_\hmu \isoto \N^f_\hlambda$ we obtain a functor $\tau = \sigma \Theta^\mu_\lambda \kappa: \O_\hlambda \to \N^f_\hlambda$.  We show in Theorem \ref{equivalentTheorem} that $\tau$ is 
naturally equivalent to $\oGf$. 

As a consequence we establish that $\oGf$ has a left and a right adjoint. (It has eluded us to directly prove the existence of these adjoints without using $\tau$.) We also prove
that $\oGf$ is a quotient functor, Corollary \ref{maincor}.

\medskip

\noindent
The idea behind the proofs of these facts is that the Whittaker functor $\oGf$ commutes not just with projective functors but actually with morphisms between projective functors in the following sense:
Let $\P$ and $\P'$ be projective endofunctors of the category of all $Z$-finite $U$-modules; in particular such functors act on $\O$ and on Whittaker modules. Let $\phi: \P \to \P'$ be a natural transformation. Then we are constructing equivalences $\Theta_\P: \P \circ \oGf \to  \oGf \circ \P $
such that $\oGf \circ \phi$ (composing a natural transformation with a functor) and $\phi \circ \oGf$ (precomposing a natural transformation with a functor) becomes equivalent after conjugation with $\Theta$:
$\Theta_{\P'} (\phi \circ \oGf) = (\oGf \circ \phi)\Theta_\P$. This is showed in Proposition \ref{commutingprop}. We also show that the functor $\tau$ commutes with projective functors and their morphisms in 
this sense. With this in hand, in order to establish Theorem \ref{equivalentTheorem} it is enough to observe that $\oGf(\D_\lambda) \cong \tau(\D_\lambda)$. 

\subsubsection{} We wonder if the commutativity between $\oGf$ and projective functors may be of interest in the study of the so called Rouqier complexes, see \cite{EW}, \cite{LW}. 
Rouquier complexes occur naturally when one constructs projective resolutions of Verma modules by iterated use of wall-crossing functors.
They are mostly interpreted as complexes of Soergel bimodules and they have been of fundamental importance in establishing Hodge theoretical properties of the category of 
Soergel bimodules. In their original form however Rouquier complexes are complexes of projective functors (and as such carry more information) and the results of this paper may perhaps help to find symmetries of them.
One may construct Rouqier complexes of Whittaker modules and it follows from the results here that these have good exactness properties, see Corollary \ref{rouquierwhittaker}. One may also use the composition of the Whittaker functor and its adjoint to provide endofunctors of $\O_\hlambda$ that commute with all projective endofunctors of $\O_\hlambda$. In section \ref{lastsection} we added some computations of the adjoints of $\oGf$ in order to facilitate this.

\subsubsection{Acknowledgements} We like to thank Paul Bressler for many useful conversations.

\section{Preliminaries. }
In this section we collect the known facts that we will need about category $\O$, Harish-Chandra bimodules, Whittaker modules and the Whittaker functor.

\subsection{Root data}
Let $\g \supset \b \supset \h$ be a complex semisimple Lie
algebra containing a Borel and a Cartan subalgebra and put $\n := [\b,\b]$. Let $W$ be the
Weyl group and $\Bb \subset \h^*$ the simple roots and  $\rho$ the half sum of the
positive roots. We consider the $dot$-action of $W$ on $\h^*$ defined by $w \cdot \lambda = w(\lambda+ \rho)-\rho$; we let $S(\h)^{W}$ be the invariants with respect to the dot-action.
Let $w_0 \in W$ be the longest element.

Let $U = U(\g)$ be the universal
enveloping algebra of $\g$ and $Z  \subset U$ its center.
For $\lambda \in \h^*$ let $J_{\lambda} \subset Z$ be the maximal ideal that annihilates a Verma module with highest weight $\lambda$.
Let $W_\lambda = \{w \in W| w \cdot \lambda = \lambda\}$. Let $W^\lambda$ be the set of longest representatives of the left cosets $W/W_\lambda$ and ${}^\lambda W$ the set of longest representatives of the right cosets
$W_\lambda \setminus W$.

\medskip

\noindent
Let $f: U(\n) \to \C$ be an algebra homomorphism. 
Let $\Bb_f = \{\alpha \in \Bb | f(E_\alpha) \neq 0\}$, where $E_\alpha \in \n$ is the Chevalley generator corresponding to $\alpha \in \Bb$, and let $W_f$ be the subgroup of $W$ generated by simple reflections $s_\alpha$, $\alpha \in \Bb_f$.

 We say that $f$ is \emph{non-degenerate} if $\Bb_f = \Bb$ and that $f$ is \emph{trivial} if $\Bb_f = \emptyset$.

\subsubsection{}
Let $\g$-$\mod$ denote the category of all (left) $U$-modules and let $\F$ be the category of finite dimensional $U$-modules. Let $\M$ denote the category of finitely generated $U$-modules such that the action of the subalgebra 
$Z \subset U$ is locally finite. For $\lambda \in \h^*$ define full subcategories of $\M$: 
$$\M_\lambda = \{V \in \M| J_\lambda V= 0\}, \ \M_\hlambda= \{V \in \M | \, \exists n>0: \, J^n_\lambda V = 0\}.
$$
We let $i_\lambda: \M_{\hlambda} \to \M$ denote the inclusion functor and define 
 $$pr_\lambda: \M \to \M_{\hlambda}, \ M \mapsto \{m \in M| \, \exists n >0 : \, J^n_\lambda \cdot m = 0\}.$$
 Then we have the block decomposition $\M = \oplus_{\lambda \in \h^*} \M_\hlambda$. For any full subcategory $\mathcal{C}$ of $\M$ we define full subcategories
 $\mathcal{C}_\lambda = \mathcal{C} \cap \M_\lambda$ and  $\mathcal{C}_\hlambda = \mathcal{C} \cap \M_\hlambda$.

\subsection{Category $\O$.}
See \cite{H}.
Let
$\O$ be the BGG-category of finitely generated $U$-modules
which are locally finite over $\n$ and semisimple over $\h$.  Any $M \in \O$ thus has the weight space decomposition $M = \oplus_{\nu \in \h^*} M^\nu$ and each $M^\nu$ is a finite dimensional vector space.
$\O$ is a full subcategory of $\M$.  Also, let $\O' \supset \O$ be the category of finitely generated $U$-modules
which are locally finite over $\b$ and locally finite over $Z$.  

For $\lambda \in \h^*$ we get full subcategories $\O_\lambda \subset O_\hlambda \subset \O$ and $\O'_\lambda \subset O'_\hlambda \subset \O'$ and block decompositions
$\O = \oplus_{\lambda \in \h^*} \O_\hlambda$ and $\O' = \oplus_{\lambda \in \h^*} \O'_\hlambda$.


For $M= \oplus_{\nu \in \h^*} M^\nu \in \O$ we let $M^* = \oplus_{\nu \in \h^*} \Hom_\C(M^\nu,\C)$ with the $\g$-module structure given by the Chevalley involution:
$(xf)(m) = f(x^{tr}m)$ for $x \in \g$, $f \in M^*$ and $m \in M$. Then $M^* \in \O$ and $M \isoto M^{**}$.

\medskip

For $\lambda \in \h^*$ denote by $\C_\lambda$ the corresponding one dimensional representation of $\b$ (by means of the projection $\b \to \h$).
Let $\D_\lambda = U(\g) \otimes_{U(\b)} \C_\lambda$ be the Verma module with highest weight $\lambda$, $\nabla_\lambda$ its dual, $L_\lambda$ its simple quotient and $P_\lambda$ a projective cover of $L_\lambda$ in $\O_\hlambda$ and $I_\lambda = P^*_\lambda$ an injective hull of $L_\lambda$.

\medskip

\noindent
Assume that a module $M \in \O$ has a filtration $0  = F_0 \subset F_1 \subset \ldots \subset F_n = M$ such that $F_i/F_{i-1} \cong \D_{\lambda_i}$ for  each $i$. Then we say that $M$ has a Verma flag
and we write $(M: \D_\lambda) = \#\{i | \, \lambda_i = \lambda\}$ for the corresponding Verma flag multiplicity. For any $V \in \O$ we write $[V:L]$ for the Jordan H\"older-multiplicity of $L $ in $V$, for $L$ a simple module.
Any projective module admits a Verma flag and BGG-reciprocity states $(P_\lambda: \D_\mu) = [\D_\mu: L_\lambda] $.

\subsection{Whittaker modules and functor.}
Let $\N$ be the category of finitely generated $U$-modules 
 which are locally finite over $U(\n)$ and locally finite over
$Z$. This is a finite length category. Let  $f: U(\n) \to
\C$ be an algebra homomorphism and denote by $\C_f$ the corresponding one dimensional representation of $U(\n)$.

Let $\N^f$ be the full subcategory of $\N$ whose objects are
locally annihilated by some power of $\Ker f$. Objects of $\N^f$ are called Whittaker modules.
(Thus, if $f$ is the trivial homomorphism then $\N^f =
\O'$.) Since $\n$ is nilpotent we have by \cite{Bourbaki}
$$\N = \oplus_{f:U(\n) \to \C} \,\N^f.$$
The categories
$\N$ and $\N^f$ also decomposes over the center $Z$:
$$
\N = \oplus_{\lambda \in \h^*} \N_{\hlambda}, \; \; \N^f =
\oplus_{\lambda \in \h^*} \N^f_{\widehat{\lambda}}.
$$

\begin{lemma}\label{Gabberlemma} a) Let $M$ be a $\g$-module and assume that each $m \in
M$ is annihilated by some power of $E_\alpha - f(E_\alpha)$ for
$\alpha \in \Bb$. Then $M$ is locally finite over $\n$. In
particular, if $M$ is finitely generated over $U$ then $M \in
\N^f$.

b)  For $E \in \F$ and $M \in \N^f$ we have $E
\otimes M \in \N^f$.
\end{lemma}
\begin{proof} a) Define a new $\n$-action on $M$ by $x * m =
(x - f(x))m$. Then the generators $E_\alpha$ of $\n$ acts
nilpotently on $m$. By \cite{BK}, Lemma 7.3.7. this implies that $U(\n)*m$ is finite dimensional for each $m \in M$. 

\medskip

\noindent
b) Let $e \otimes m \in E \otimes M$. Then
$(E_\alpha-f(E_\alpha))(e \otimes m) = E_\alpha e \otimes
(E_\alpha-f(E_\alpha))m$ and therefore by induction
$$(E_\alpha-f(E_\alpha))^n(e \otimes m) = \sum^n_{j=0}
\binom{n}{j}E^j_\alpha e \otimes (E_\alpha-f(E_\alpha))^{n-j}m.$$
Since $E$ is finite dimensional $E^j_\alpha e = 0$ for $j>>0$ and
since also $(E_\alpha-f(E_\alpha))^{n-j}m = 0$ for $n-j>>0$ we get
that the above sum vanishes for $n>>0$. Thus b) follows from a).
\end{proof}

\subsubsection{Standard Whittaker modules}\label{xxxyyy} See \cite{McD}, \cite{MS}. Let $f$ be fixed and
let $\p$ denote the parabolic subalgebra of $\g$ generated by $\b$ and $E_\alpha$ for $\alpha \in \Bb_f$. Let $\l$ be the reductive Levi factor of $\p$ and put  
$J^{(\l)}_\lambda = \Ann_{Z(\l)} U(\l) \otimes_{U(\b \cap \l)} \C_\lambda$, for $\lambda \in \h^*$.

Consider the $U(\l)$-module $U(\l)/J^{(\l)}_\lambda \otimes_{U(\n \cap \l)} \C_f$ as a $U(\p)$-module by means of the projection $\p \to \p/rad\, \p \cong \l$ and define the standard Whittaker module
$$
\D_\lambda(f) = U \otimes_{U(\p)} (U(\l)/J^{(\l)}_\lambda \otimes_{U(\n\cap \l)} \C_f), \; \lambda \in \h^*.
$$
Note that when $f$ is the trivial homomorphism then  $\D_\lambda(f) =\D_\lambda$. $\D_\lambda(f) $ has a unique irreducible quotient
$L_\lambda(f) $.

Also, define $$\D^n_\lambda(f) = U \otimes_{U(\p)} (U(\l)/(J^{(\l)}_\lambda)^n \otimes_{U(\n \cap \l)} \C_f)$$
for $n \geq 1$.  

\subsection{The Whittaker functor}
For a $\g$-module $V$ we define
$$\Gamma_f(V) = \{v \in V| (\Ker f)^nv = 0,
\operatorname{for} n >> 0\}$$ and 
$$Wh_f(V) = \{v \in V| (\Ker f)v = 0\}.$$ Then $\Gamma_f(V)$ is a $\g$-submodule of $V$ while $Wh_f(V)$ is merely a $Z$-submodule.
It is clear that $\Gamma_f$ and $Wh_f$ yield functors. $Wh_f$ was introduced by Kostant, \cite{K}.

For $M = \oplus_{\lambda \in \h^*} M^\lambda \in \O$ 
we
define the completion $\overline{M} = \prod_{\lambda \in \h^*}
M^\lambda$. This has a natural $\g$-module structure making $M
\subseteq \overline{M}$ a submodule. Let $Z-\mod_{fd}$ denote the category of finite dimensional $Z$-modules. In \cite{Bac} the second author introduced the functors
$$
\overline{\Gamma}_f : \O \to \N^f, \ M \mapsto  \Gamma_f(\overline{M}) \hbox{ and }
$$
$$
\overline{Wh}_f : \O \to Z-\mod_{fd}, \ M \mapsto  Wh_f(\overline{M}).
$$
\emph{$\oGf$ is called the Whittaker functor}; it is exact for any $f$. The functor $\overline{Wh}_f$ is exact if and only if $f$ is non-degenerate. These functors commute with the action of $Z$.
Assume that $f$ is non-degenerate:  \cite{K} (see also \cite{MS1} for a geometric proof) showed that the functor
$$
Wh_f: \N^f \to Z-\mod_{fd}
$$
is an equivalence of  categories; its quasi-inverse is $M \mapsto U\otimes_{U(\n) \otimes Z} M$. Here the left $Z$-module structure on $M$ is the given one and the $U(\n)$-module structure is the unique one such that
$\Ker f \cdot M = 0$. 

In \cite{Bac} it was proved that
\begin{proposition}\label{mythesis} For any $\lambda \in \h^*$ we have 
$\overline{\Gamma}_f(\D_\lambda) \cong \D_\lambda(f)$. If, moreover, $\lambda$ is dominant and $x \in {}^\mu W $ then
$\overline{\Gamma}_f(L_{x \cdot \lambda}) \cong L_{x \cdot \lambda}(f)$ and for  $x \notin {}^\mu W$ we have $\overline{\Gamma}_f(L_{x \cdot \lambda})=0$.
\end{proposition}
\subsubsection{} Let $\lambda$ be integral and dominant. 
Let $J = \Ann_{Z}(P_{w_0\cdot \lambda})$ and put $C = Z(g)/J$.
Let $\V: \O_{\hlambda} \to Z-\mod$ be Soergel's functor $\V(M)
= \Hom_{\O}(P_{w_0\cdot \lambda}, M)$ \cite{Soergel}. Thus, $C = \V(P_{w_0\cdot
\lambda})$. Since $\V$ is fully faithful on projective objects we
conclude that $J\O_{\hlambda} = 0$. Thus, also $J
\overline{\Gamma}_f(M)= J \overline{Wh}_f(M) = 0$ for $M \in
\O_{\hlambda}$.
The following result was proved in \cite{Bac}. We include here a short proof that Soergel once explained to us.
\begin{proposition}\label{aprop} Assume that $f$ is non-degenerate. Then $\overline{Wh}_f |_{\O_\hlambda}$ is equivalent to $\V$.
\end{proposition}
\begin{proof} The assumption on $f$ implies that  $\overline{Wh}_f$ is exact. Let $v \in \overline{Wh}_f(P_{w_0\cdot \lambda})$ be
such that $\overline{v} \neq 0$ in $\overline{Wh}_f(L_{w_0\cdot
\lambda})$ under the surjection $\overline{Wh}_f(P_{w_0\cdot
\lambda}) \twoheadrightarrow \overline{Wh}_f(L_{w_0\cdot \lambda})$ induced by
the surjection $P_{w_0 \cdot \lambda} \twoheadrightarrow L_{w_0\cdot \lambda}$. By Yoneda Lemma
$$\Hom_{functors}(\V,\overline{Wh}_f) = \Hom_{Z}(\V(P_{w_0\cdot
\lambda}),\overline{Wh}_f(P_{w_0\cdot \lambda})) =$$
$$
= \Hom_{Z}(C,\overline{Wh}_f(P_{w_0\cdot \lambda})) =
\Hom_{C}(C,\overline{Wh}_f(P_{w_0\cdot \lambda}))
=\overline{Wh}_f(P_{w_0\cdot \lambda}).
$$
Let $h: \V \to \overline{Wh}_f$ be the natural transformation
that corresponds to $v \in \overline{Wh}_f(P_{w_0\cdot \lambda})$.
Then we see that $h_{L_{w_0\cdot \lambda}}: \V(L_{w_0\cdot
\lambda}) \to   \overline{Wh}_f(L_{w_0\cdot \lambda})$ is an
isomorphism since it is non-zero and both sides are one
dimensional vector spaces. Also, since $\V(L_{x\cdot \lambda}) =
\overline{Wh}_f(L_{x\cdot \lambda}) = 0$ for $x \neq w_0$ we get
that  $h_{L_{x \cdot \lambda}}$ is an isomorphism as well.
Hence, by exactness of both functors and the five lemma
$h_{M}$ is an isomorphism for all $M \in \O_{\hlambda}$.
\end{proof}
\subsection{Projective functors}\label{Projective functors Section}
Let $E \in \F$ and $M \in \M$; then there is the functor
$$
T_E: \M \to \M, \ T_E(M) = E \otimes M.
$$
A projective functor from $\M_\hlambda$ to $\M_\hmu$ is a direct summand in a functor $pr_\mu T_E i_\lambda$. Projective functors are exact and they have left and right adjoints which coincide.

By \cite{BG} we have the following important: Assume that $\lambda$ is dominant and let $\P, \P': \M_\hlambda \to \M_\hmu$ be projective functors.
Let $g \in \Hom(\P(\D_\lambda), \P'(\D_\lambda))$. Then there is a natural transformation $\phi: \P \to \P'$ such that $\phi(\D_\lambda) = g$. Moreover, $\phi$ is uniquely determined by $g$ if $g$ is an isomorphism and if $g$ is an idempotent then we can chose $\phi$ to be an idempotent as well. Thus decomposing $\P(\D_\lambda)$ into a direct sum of indecomposables and decomposing $\P$ is the same thing.

Assume that $\lambda, \mu \in \h^*$ such that $\mu-\lambda$ is integral and  $W_\lambda \subseteq W_\mu$. Let $V$ be a finite dimensional irreducible $\g$-module with extremal weight $\mu-\lambda$. Then there is the \emph{translation to the wall functor}
\begin{equation}\label{ontowall}
\Theta^\mu_\lambda: \M_{\hlambda} \to \M_{\hmu}, \; M \mapsto pr_\mu(T_V(M)).
\end{equation}
Its left and right adjoint is \emph{translation out of the wall};
\begin{equation}\label{outofwall}
\Theta^\lambda_\mu: \M_{\hmu} \to \M_{\hlambda}, \; M \mapsto pr_\lambda(T_{V'}(M)).
\end{equation}
Here $V'$ is a finite dimensional irreducible $\g$-module with extremal weight $\lambda-\mu$. 
We have $\Theta^\mu_\lambda \Theta^\lambda_\mu \cong Id^{|W_\mu|}_{\M_\hmu}$.

\subsubsection{Projective functors on $\O$}\label{transOnO} 
Any projective functor $\P: \M_\hlambda \to \M_\hmu$ descends to functors $\P: \O_\hlambda \to \O_\hmu$ and $\P: \O'_\hlambda \to \O'_\hmu$.  Those are the projective functors on $\O$ and $\O'$;
they are exact, maps projectives to projectives and commute with duality: $\P(M^*) \cong \P(M)^*$.

Let $\lambda$ and $\mu$ be as in Section \ref{Projective functors Section}.  Thus we get translation to and out of the wall:
$\Theta^\mu_\lambda: \O_{\hlambda} \to \O_{\hmu}$
and 
$\Theta^\lambda_{\mu}: \O_{\hmu} \to \O_{\hlambda}$.

We have
$$
\Theta^\mu_\lambda \D_{x \cdot \lambda} \cong \D_{x \cdot \mu}, \, x \in W,
$$
$$
\Theta^\mu_\lambda L_{x \cdot \lambda} = L_{x \cdot \mu}, \, x \in W^\mu ,
$$
and
$\Theta^\mu_\lambda L(x \cdot \lambda)  = 0$,  for $x \notin W^\mu$.

For any projective module $P \in \O_\hlambda$ there is a unique (upto isomorphism) projective functor $\P: \O_\hlambda \to \O_\hlambda$ such that $\P(\D_\lambda) = P$.

\subsubsection{Projective functors on $\N^f$}
By Lemma \ref{Gabberlemma} any projective functor $\P: \M_\hlambda \to \M_\hmu$ descends to a functor $\P: \N^f_\hlambda \to \N^f_\hmu$.  Those are our projective functors on $\N^f$. 
\footnote{Contrary to the case of $\O$ it is a priori not clear whether an indecomposable projective functor on $\M$ when restricted to $\N^f$ remains indecomposable. 
However we only need those projective functors on $\N^f$ that are restrictions of projective endofunctors of $\M$.}

\subsection{Harish-Chandra bimodules and Soergel-Mili\v{c}i\'{c}'s equivalence.}
Throughout this section we fix dominant integral weights $\lambda$ and $\mu$ with $\lambda$ regular. \footnote{One could weaken the integrality condition to $\lambda- \mu$ is integral here. But for the sake of simplicity we have assumed both $\lambda$ and $\mu$ are integral.}
\subsubsection{}
Let $X$ be a $U-U$-bimodule. Then we have the adjoint action $ad$
of $\g$ on $X$ given by $ad(g)x = gx-xg$ and the sub-bimodule $X_{adf} \subseteq X$ consisting of ad-finite vectors. If $X$ is ad-finite (i.e. if $X = X_{adf}$) then $X$ is finitely generated as a bimodule iff X is finitely generated as a left module iff $X$ is finitely generated as a right module.

The category of Harish
Chandra bimodules $\H$ is the category of finitely generated $U-U$-bimodules
which are locally finite with respect to the adjoint action of
$\g$ and to the left (or equivalently the right) action of the center $Z$. The category of Harish-Chandra bimodules
decomposes into blocks ${}_{\hlambda}\H_{\hmu} := {}_{\hlambda}\H \cap \H_{\hmu}$, for $\lambda, \mu \in \h^*$, where
$$
{}_{\hlambda}\H = \{X \in \H | \; \exists n > 0, J^n_{\lambda} X = 0\},
$$
$$
\H_{\hmu} = \{X \in \H | \; \exists n > 0, X J^n_{\mu}  = 0\}.
$$
Similarly we define ${}_{\lambda}\H , \H_\mu, {}_{\lambda} \H_{\mu}$ and
${}_{\hlambda} \H_{\mu}$.
 There is an autoequivalence $V \mapsto s(V) =: V^s$ of $\H$ where $V^s = V$ as a set and the $U-U$-bimodule action is given by $u * v * u' := (u')^tvu^t$ for $v \in V^s$. 
 Since the Chevalley involution fixes $Z$ we have  $({}_{\hlambda}\H_{\hmu})^s =  {}_{\hmu}\H_{\hlambda}$.

 \medskip

 \noindent
 For $E \in \F$ we consider the $U-U$-bimodule $E^l = E$ as a set and with action $u * e * u' = ue$ and the $U-U$-bimodule $E^r = E$ as a set and action
 $u * e * u' = (u')^t e$.  Note that $(E^l)^s = E^r$.

 Consider the canonical projections $(pr_\mu)^l: \H \to {}_{\hmu}\H$ and $(pr_\mu)^r: \H \to \H_{\hmu}$. Since the left and right $U$-action commute
 we see that  $(pr_\mu)^l(\H_{\lambda}) = {}_{\hmu}\H_{\lambda}$ $(pr_\mu)^l(\H_{\hlambda}) = {}_{\hmu}\H_{\hlambda}$ and similarly for $(pr_\mu)^r$. If $V$ is any finite dimensional
 bimodule and $X \in \H$ then $V \otimes X \in \H$ with the canonical bimodule structure. Note that $(V \otimes X)^s = V^s \otimes X^s$.
 
 Similarly, if $\P$ is a projective functor on $\M$ we get projective functors $\P^l, \P^r: \H \to \H$.

\subsubsection{Equivalences with category $\O$}
 By results of Bernstein and Gelfand \cite{BG}, Soergel \cite{S} and Soergel and Mili\v{c}i\'{c}  \cite{MS} we have mutually inverse equivalences
 $$
 F_{\mu}: \O'_{\hmu} \leftrightarrows {}_{\hmu}\H_{\hlambda}: G_{\mu}
 $$
 where $F_{\mu}(X) = {\varinjlim}_n \hlambdam_\C(\D^n_\lambda, X)_{adf}$ and $G_{\mu}(Y) = {\varprojlim}_n Y \otimes_U \D^n_\lambda$. 
 (Here the $U$-bimodule structure on $\Hom_\C(\D^n_\lambda, X)$ is given by $(ufu')(m) = u\cdot f(u'm)$.)

 These functors restrict to equivalences
  $$
 F_{\mu}: \O_{\hmu} \leftrightarrows {}_{\hmu}\H_{{\lambda}}: G_{\mu}
 $$
  where $F_{\mu}(X) =  \hlambdam_\C(\D_\lambda, X)_{adf}$ and $G_{\mu}(Y) = Y \otimes_U \D_\lambda$ which in turn restrict to equivalences
 $$
 F_{\mu} : \O_{\mu} \leftrightarrows {}_{\mu}\H_{{\lambda}}: G_{\mu}.
 $$
 Let $\kappa: \O_\hlambda \isoto \O'_\lambda$ denote the equivalence which is defined as the composition
 \begin{equation}\label{kappa}
\kappa:  \O_\hlambda \overset{F_\lambda}{\to} {}_{\hlambda}\H_{\lambda} \overset{s}{\to} {}_{\lambda}\H_{\hlambda} \overset{G_\lambda}{\to} \O'_\lambda.
\end{equation}
Then we have
\begin{equation}\label{kappa1}
\kappa(\D_{x \cdot \lambda}) \cong \D_{x^{-1} \cdot \lambda}, \ x \in W.
\end{equation}
 
 \subsubsection{Soergel Mili\v{c}i\'{c}'s equivalence}
Assume now that  $W_f = W_\mu$. Soergel and Mili\v{c}i\'{c}  \cite{MS} constructed the equivalence
$$
G_f:{}_{\widehat{{\lambda}}} \H_{\hmu} \to
\N^f_{\widehat{{\lambda}}},\ X \mapsto {\varprojlim}_n X \otimes_U
\Delta^n_\mu(f).
$$
Thus we get the equivalence $\sigma = G_f \circ s \circ F_\mu: \O_{\hmu}
\to \N^f_{\hlambda}$. It is known that
\begin{equation}\label{sigma1}
\sigma(\D_{x \cdot \mu}) \cong \D_{x^{-1} \cdot \lambda}(f),\;
x \in W.
\end{equation}
 The following lemma is easy to prove, for details
see \cite{J}. 
\begin{lemma}\label{commutelemmaA} For $M \in \O_{\lambda}$ we have a natural isomorphism $E^l \otimes F_{\lambda}(M) \isoto F_{\lambda}(E \otimes M)$, $e \otimes \phi \mapsto
\{x \mapsto e \otimes \phi(x)\}$. It induces an isomorphism $(pr_\lambda)^l(E^l \otimes F_{\lambda}(M)) \isoto F_{\lambda}(pr_\lambda(E \otimes M))$. Similarly, we have
a natural isomorphism $(\Theta^\mu_\lambda)^l F_{\lambda}(M) \cong F_{\mu}\Theta^\mu_\lambda(M)$.
\end{lemma}

\section{The main results.}
Throughout this section we assume that
$\lambda$ is a regular dominant integral weight and $\mu$ a dominant integral weight. Let $f: U(\n) \to \C$ be a character such that $\Bb_f = \Bb_\mu$.

\begin{proposition}\label{commutingprop}   a.) For any projective functor $\P: \M \to \M$ there is a natural isomorphism $\Theta_\P:  \P \circ \oGf \to   \oGf \circ \P $ between functors from
$\O$ to $\N^f$.
b) For any morphism $\phi: \P \to \P'$ of projective functors we have a commutative diagram
\[ \xymatrix{ \P  \circ  \oGf \ar[rr]^{  \phi   \circ \oGf } \ar[d]_{\Theta_\P} &  &  \P'  \circ \oGf \ar[d]^{\Theta_{\P'}} \\  \oGf  \circ  \P  \ar[rr]_{  \oGf \circ \phi  } & &\oGf\circ   \P'   }  \]
\end{proposition}
\begin{proof}  It is enough to prove $a)$ and $b)$ when $\P = T_E$ and $\P' = T_{E'}$ because then it follows that $a)$ also holds for any direct summand in $T_E$ by taking $\phi: T_E \to T_E$ to be an orthogonal projection onto this summand. Then the general case for $b)$ follows as well.

\medskip \noindent We prove a) for $\P = T_E$. Since $E$ is finite dimensional we have
$\overline{E \otimes M} = E \otimes \overline{M}.$ Therefore
$$\overline{\Gamma}_f(E \otimes M) = \Gamma_f(\overline{E \otimes
M}) =\Gamma_f(E \otimes \overline{M}).$$

By Lemma \ref{Gabberlemma} b) we conclude that $E
\otimes \Gamma_f(\overline{M}) \subseteq \Gamma_f(E \otimes
\overline{M})$. This inclusion is denoted $\Theta_{T_E}(M)$. In order to prove that it is an
isomorphism, we proceed as follows: Let ${e}_1, \ldots {e}_k$ be a basis
for $E$ where ${e}_i$ is a weight vector of weight $\lambda_i$ ordered in such a way that
$\lambda_i < \lambda_j$ implies $i < j$. Let $\sum^k_{i=1} {e}_i
\otimes m_i \in \Gamma_f(E \otimes \overline{M})$. We must prove that each $m_i \in \Gamma_f(\ov{M})$. Pick $n = n_1 >0$ such 
that for all $\alpha \in \Bb$ we have
$$
0 = (E_\alpha-f(E_\alpha))^{n}(\sum^k_{i=1} {e}_i \otimes m_i) =
\sum^k_{i=1} \sum^n_{j=0} \binom{n}{j}E^j_\alpha {e}_i \otimes
(E_\alpha-f(E_\alpha))^{n-j}m_i.
$$
Note that $\operatorname{Span}_\C\{e_2, \ldots , e_k\}$ is a $U(\n)$-submodule of $V$. Therefore
it follows that $e_1 \otimes (E_\alpha-f(E_\alpha))^{n}m_1= 0$ and hence that
$(E_\alpha-f(E_\alpha))^{n}m_1= 0$.

Note that $E^{k}_\alpha {e}_i = 0$ for all $i$.
Let $n_2 = n+k$. Then $(E_\alpha-f(E_\alpha))^{n_2}(e_1 \otimes m_1)=0$ and therefore
$$(E_\alpha-f(E_\alpha))^{n_2}(\sum^k_{i=2} {e}_i \otimes m_i) = 0.
$$
By repeating the above argument we conclude that  $(E_\alpha-f(E_\alpha))^{n_2} m_2 =0$.
Proceeding by induction we conclude that $m_i \in
\Gamma_f(\overline{M})$ for all $i$.  

\medskip

\noindent
We now prove $b)$ for $\P = T_E$ and $\P' = T_{E'}$. Let $M \in \O$ and $\phi_M: T_E(M) \to T_{E'}(M)$ be the morphism given by $\phi: T_E \to T_{E'}$. We must show that the following diagram commutes:

\[ \xymatrix{ T_E\oGf (M)\ar[rr]_{\phi_{\oGf (M)}}   \ar[d]_{\Theta_{T_E}} &  & T_{E'}\oGf (M)\ar[d]^{\Theta_{T_{E'}}} \\  \oGf T_E(M) \ar[rr]^{\oGf {\phi_{M}}}   & &  \oGf T_{E'}(M)  }  \]

Let ${e}_1, \ldots {e}_k$ and ${e}'_1,\ldots,  {e}'_l$ be bases of  $E$ and $E'$. Let $\ov{U}= U/\mathfrak{a}$ where $\mathfrak{a}\subset Z$ is an ideal of finite codimension such that $\mathfrak{a} M = \mathfrak{a}T_E(M) = 0$. Then automatically $\mathfrak{a} \oGf(M)=0$.
Pick  $u_{ij} \in \ov{U}$ such that $\phi_{\ov{U}}({e}_i \otimes 1) = \sum {e}'_j \otimes u_{ij}$. Then $\phi_A({e}_i \otimes a) = \sum {e}'_j \otimes u_{ij} a$ for any $A \in \ov{U}$-$\mod$, by functoriality. Hence we get for ${e}_i \ot m \in T_E\oGf (M)$ that
$$\oGf {\phi_{M}}(\Theta_{T_E}({e}_i \ot m)) =  \oGf {\phi_{M}}({e}_i \ot m) =  
$$
$$\sum {e}'_j \ot u_{ij}m =\phi_{\oGf (M)}({e}_i \ot m) = \Theta_{T_{E'}}(\phi_{\oGf (M)}({e}_i \ot m)).  
$$
\end{proof}
A weaker version of a) above was proved in \cite{MR3994467} Proposition 2.3.4. He showed that $\oGf(T_E(\D_\lambda)) \cong T_E(\oGf(\D_\lambda))$.

\subsubsection{}

Let $s$ be a simple reflection and assume that $W_\mu = \{e,s\}$. Then there is the wallcrossing functor $\Theta_s := \Theta^\lambda_\mu \Theta^\mu_\lambda: \O_\hlambda \to \O_\hlambda$.
Let $\Psi_s$ denote the complex of functors $Id \to \Theta_s$ (given by the adjunction morphism). Then if $w=s_1 \cdots s_m$ is a reduced expression of an element in $W$ we get the Rouqier complex functor
$\Psi_w = \Psi_{s_1} \circ \cdots \circ \Psi_{s_m}$, see \cite{LW}.
\begin{cor} \label{rouquierwhittaker} Let $x \in W$ and let $s$ be a simple reflection such that $x \prec xs$. Then $\Theta_s(\D_{x \cdot \lambda}(f)) = \Theta_s (\D_{xs \cdot \lambda}(f))$ and the adjunction morphisms 
$Id \to \Theta_s$ and $\Theta_s \to Id$ define a short exact sequence $$0 \to \D_{x\cdot \lambda}(f) \to \Theta_s  (\D_{x \cdot \lambda}(f)) \to   \D_{xs \cdot \lambda}(f) \to 0.$$
Hence the Rouquier complex $\Psi_w(\Delta_{x \cdot \lambda}(f))$ is an exact resolution of its $0$'th cohomology for any $x \in W$.
\end{cor}
\begin{proof} It is well-known that the sequence $0 \to \D_{x\cdot \lambda} \to \Theta_s  (\D_{x \cdot \lambda}) \to   \D_{xs \cdot \lambda} \to 0$ is exact. Applying the functor $\oGf$ we get from Propositions \ref{mythesis} 
and \ref{commutingprop} the first two statements of the lemma. This implies formally that the above Whittaker Rouquier complex is exact as well.
\end{proof}

\subsection{The functor $\tau:\O_{\hlambda} \to   \N^f_{\hlambda}$}
Recall the equivalences $$\kappa: \O_{\hlambda} \overset{F_{{\lambda}}}{\longrightarrow} {}_{\hlambda}\H_{{\lambda}} \overset{s}{\longrightarrow} {}_{{\lambda}}\H_{\hlambda}  \overset{G_{{\lambda}}}{\longrightarrow}
\O'_{\lambda} \hbox{ and }\sigma:  \O'_{\hmu}
\overset{F_{\mu}}{\longrightarrow} {}_{\hmu}\H_{\hlambda}
\overset{s}{\longrightarrow} {}_{\hlambda}\H_{\hmu}
\overset{G_f}{\longrightarrow} \N^f_{\hlambda} .$$

Let $\tau$ denote the composition 
\begin{equation}\label{deftau}
\tau = \sigma \circ \Theta^\mu_\lambda \circ  \kappa: \O_{\hlambda}  \overset{\kappa}{\longrightarrow} \O'_{\lambda}
\overset{\Theta^\mu_{\lambda}}{\longrightarrow} \O'_{\hmu} \overset{\sigma}{\longrightarrow}  \N^f_{\hlambda}. 
\end{equation}


It is evident that $\tau$ is exact. We shall show in Theorem \ref{equivalentTheorem} that $\tau$ is equivalent to $\oGf$. Since $\kappa$ and $\sigma$ are equivalences $\tau$ is in reality
determined by the projective functor $\Theta^\mu_{\lambda}: \O'_{\lambda} \to \O'_{\hmu}$.  We shall see in Lemma \ref{commutelemmaB} below
that the functor $\tau$ commutes with projective functors (and with morphisms of projective functors). The functor $\Theta^\mu_{\lambda}: \O'_{\lambda} \to \O'_{\hmu}$ obviously doesn't commute with projective functors, but after conjugating it with the equivalences $\sigma$ and $\kappa$ it does. The reason for this is essentially that the left and right action of $U$ on a Harish-Chandra bimodule commutes.

\medskip

\noindent
The left adjoint of the inclusion $i: \O'_{\lambda} \to \O'_{\hlambda}$ is $( \ ) \otimes_U U_\lambda$ and its right adjoint is $( \ )^{J_{\lambda}}$, the functor of taking $J_{\lambda}$-invariants. 
Therefore $\tau$ has the left adjoint 
\begin{equation}\label{tauL}
\tau^L =  \kappa^{-1} \circ  \Theta^\lambda_{\mu}(  \ ) \otimes_U U_\lambda \circ \sigma^{-1}.
\end{equation}
and the right adjoint
\begin{equation}\label{tauR}
\tau^R =  \kappa^{-1} \circ  \Theta^\lambda_{\mu}(  \ )^{J_{\lambda}} \circ \sigma^{-1}.
\end{equation}
\subsubsection{}
Expanding the maps in \eqref{deftau} we deduce a canonical equivalence of functors
\begin{equation}\label{deftau2}
\tau \cong G_f \circ (\Theta^\mu_\lambda)^r \circ F_\lambda:  \O_{\hlambda}  \overset{F_\lambda}{\longrightarrow} {}_{\hlambda}\H_{{\lambda}}  
\overset{ (\Theta^\mu_\lambda)^r }{\longrightarrow}  {}_{\hlambda}\H_{\hmu} \overset{G_f}{\longrightarrow} \N^f_{\hlambda}.
\end{equation}
This formula is easier to work with.

\begin{lemma}\label{commutelemmaB} 
{\bf a)} For any projective functor $\P: \M_\hlambda  \to \M_\hlambda$ there is a natural isomorphism $\eta_\P:  \P \circ \tau \to   \tau \circ \P $ between functors from $\O_\hlambda$ to $\N^f_\hlambda$.
{\bf b)} For any morphism $\phi: \P \to \P'$ of projective functors we have a commutative diagram
\[ \xymatrix{ \P \tau \ar[rr]^{\phi_{\tau}} \ar[d]_{\eta_\P} &  & \P' \tau \ar[d]^{\eta_{\P'}} \\ \tau\P \ar[rr]_{\tau} & & \tau \P'  }  \]

\end{lemma}

\begin{proof}  By similar arguments as were given in the beginning of the proof of Proposition \ref{commutingprop} we may assume that $\P = pr_\lambda T_E i_\lambda$ and $\P' = pr_{\lambda}T_{E'} i_\lambda$.
Let $V$ be a finite dimensional irreducible representation with extremal weight $\mu -\lambda$ so that $\Theta^\mu_\lambda = pr_\mu T_{V}$.  We use \eqref{deftau2} as the definition of $\tau$.

{\bf a)} Let $M \in \O_{\hlambda}$.   Let $\eta_\P = \eta_\P(M)$ denote the composition of the natural isomorphisms:
$$
\P\tau(M) =  \P G_f \circ (\Theta^\mu_\lambda)^r \circ F_\lambda(M) \cong  G_f \circ \P^l  \circ  (\Theta^\mu_\lambda)^r \circ F_\lambda(M) \cong
$$
$$
G_f \circ (\Theta^\mu_\lambda)^r \circ \P^l \circ F_\lambda(M) \cong G_f \circ (\Theta^\mu_\lambda)^r \circ F_\lambda(\P M) = \tau \P (M).
$$
Let us describe $\eta_\P$ explicitly.  Let $X \in \O_\hlambda$ and $Y \in  {}_{\hlambda}\H_{\hmu}$.  An element in $\P X$ is a (linear combination of elements of the form) $e \otimes x$, $e \in E$ and $x \in X$
and an element in $F_\lambda X$ can be represented by a $\C$-linear function $\psi: \D^n \to X$ (for $n$ sufficiently large). 
An element in $(\Theta^\mu_\lambda)^r Y \; ( \subseteq V^r \otimes Y)$
can be written as $v^r \otimes y$ for $v \in V$, $y \in Y$, and 
an element in
$G_f(Y)$ can be represented by $y \otimes_U \ov{1} \in Y \otimes_U \D^n_\mu(f)$ (for $n$ sufficiently large). Here $\ov{1} \in \D^n_\mu(f)$ is a generator.

Therefore an element in $\P\tau(M)$ can be written as (a linear combination of elements)
$$e \otimes ((v^r \otimes \psi) \otimes_U \ov{1}),$$ 
where $\psi \in \Hom_\C( \D^n \to M)$, $e \in E$ and $v \in V$.  We then have 
\begin{equation}\label{etaPacts}
\eta_\P(e \otimes ((v^r \otimes \psi) \otimes_U \ov{1})) = (v^r \otimes (e^l \otimes \psi)) \otimes_U \ov{1},
\end{equation}
where $e^l \otimes \psi \in \Hom(\D_\lambda, E \otimes M)$ denotes the function $d \mapsto e \otimes \psi(d)$.

\medskip 

\noindent   \textbf{b)}  Let $M \in \O_\hlambda$. We want to establish the commutativity of the diagram
\begin{equation}\label{commutativesquare}
\xymatrix{ \P \tau(M) \ar[rr]^{\phi_{\tau(M)}} \ar[d]_{\eta_{\P}} &  & \P' \tau(M) \ar[d]^{\eta_{\P'}} \\ \tau\P (M)\ar[rr]_{\tau(\phi_{M})} & & \tau\P'(M)  } 
\end{equation}
Let $N >0$ be such that $J^N_\lambda$ annihilates $M$, $\tau(M)$ and $\P(M)$ and define $\ov{U} = U/J^N_\lambda$.

Let ${e}_1, \ldots {e}_k$ and $e'_1,\ldots,  e'_l$ be bases of 
 $E$ and $E'$, respectively. 
Then $\phi_{\ov{U}} : \P \ov{U} \to \P' \ov{U}$ is given by $\phi_{\ov{U}}({e}_i \otimes 1) = \sum_j e'_j\otimes u_{ij}$ for some $u_{ij}\in \ov{U}$. 
Therefore we get $\phi_{A} : \P A \to \P' A$ is given by $\phi_A({e}_i\otimes a ) = \sum e'_j \otimes u_{ij}a$, for $A \in \ov{U}$-$\mod$ by functoriality.\\
We now calculate
$$
\tau (\phi_M) \circ \eta_\P(e_i \otimes ((v^r \otimes \psi) \otimes_U \ov{1})) = \tau (\phi_M) ((v^r \otimes (e^l_i \otimes \psi)) \otimes_U \ov{1}) = $$
$$
\sum_j (v^r \otimes (e'^l_j \otimes u_{ij} \psi)) \otimes_U \ov{1}.
$$
On the other hand we have
$$
\eta_{\P'} \circ \phi_{\tau(M)}(e_i \otimes ((v^r \otimes \psi) \otimes_U \ov{1})) = \sum_j \eta_{\P'} (e'_j \otimes u_{ij}((v^r \otimes \psi) \otimes_U \ov{1})) =
$$
$$
 \sum_j \eta_{\P'} (e'_j \otimes ((v^r \otimes  u_{ij}\psi) \otimes_U \ov{1})) =  \sum_j (v^r \otimes (e'^l_j \otimes u_{ij} \psi)) \otimes_U \ov{1}.
$$
\end{proof}

\begin{theorem}\label{equivalentTheorem}  $\tau$ is naturally equivalent to $\overline{\Gamma}_f$.
\end{theorem}
\begin{proof} Let $Proj(\O_\hlambda)$ denote the full subcategory of projective objects in $\O_\hlambda$. 
Let $\Ac$ be the full subcategory of $\O_\hlambda$ whose objects are of the form $\P(\D_\lambda)$, where $\P: \M_{\hlambda} \to \M_{\hlambda}$ is a projective functor. 
Then the inclusion $\Ac \to Proj(\O_\hlambda)$ is an equivalence of categories. Let $P = \P(\D_\lambda), P' = \P'(\D_\lambda) \in \A$ and let 
$g: P \to P'$ be a morphism. Then according to \cite{BG} there is a natural transformation $\phi : \P \to \P'$ (not necessarily unique) such that $\phi_{\D_\lambda} = g$. 
By Proposition \ref{mythesis} we have $\oGf(\D_\lambda) \cong \D_\lambda(f)$ and using \eqref{kappa1}, \eqref{sigma1} and $\Theta^\mu_\lambda(\D_\lambda) \cong \D_{\mu}$ we also obtain $\tau (\D_\lambda) \cong \D_\lambda(f).$
Fix an isomorphism  
$$\epsilon :  \oGf (\D_\lambda) \isoto \tau( \D_\lambda).$$ 
Consider the diagram
\[\xymatrixcolsep{5pc}
\xymatrix{
\oGf \P({\D_{\lambda}}) \ar[r]^{\overline{\Gamma}_f(\phi_{{\D_{\lambda}}})} \ar[d]_{\Theta^{-1}_{\P}}& \overline{\Gamma}_f \P'({\D_{\lambda}})  \ar[d]^{\Theta^{-1}_{\P'}}\\
\P \overline{\Gamma}_f ({\D_{\lambda}})  \ar[r]^{\phi_{\overline{\Gamma}_f{\D_{\lambda}}}}  \ar[d]_{\P(\epsilon)}& \P' \overline{\Gamma}_f ({\D_{\lambda}}) \ar[d]^{\P'(\epsilon)} \\
\P \tau ({\D_{\lambda}})  \ar[r]^{\phi_{\tau{\D_{\lambda}}}}   \ar[d]_{\eta_{\P}} & \P' \tau ({\D_{\lambda}})  \ar[d]^{\eta_{\P'}} \\
\tau\P ({\D_{\lambda}}) \ar[r]^{\tau(\phi_{{\D_{\lambda}}})}& \tau\P' ({\D_{\lambda}})}
\]
The middle square is obviously commutative and by Proposition \ref{commutingprop} and Lemma \ref{commutelemmaB} the top and the bottom squares are commutative as well.
Thus the outer square is commutative and therefore we can define 
the natural transformation $\psi: \oGf |_\Ac \to \tau |_\Ac$ by $$\psi_P = \eta_\P \circ \P(\epsilon) \circ \Theta^{-1}_\P(\Delta_\lambda): \oGf( P) \to \tau(P).$$  
Evidently $\psi$ is an equivalence and it induces an equivalence between the induced functors
$$
\oGf, \tau: D^b(\O_\hlambda) \cong K^b(Proj(\O_{\hlambda})) \cong K^b(\Ac) \to D^b(\N^f_{\hlambda}).
$$
Since $\oGf$ and $\tau$ are exact and $\O_\hlambda$ is the heart of the standard $t$-structure on $D^b(\O_\hlambda)$ the last equivalence restricts to an equivalence between the original functors
$\oGf, \tau : \O_\hlambda \to \N^f_{\hlambda}$.
\end{proof}
\begin{remark}  As a special case we see that Soergel's functor $\V$ is equivalent to $\sigma \Theta^{-\rho}_\lambda \kappa$. This follows from Theorem \ref{equivalentTheorem} and Proposition \ref{aprop}.
\end{remark}

\begin{cor}\label{maincor} The functor $\oGf: \O_\hlambda \to \N^f_{\hlambda}$ has a left adjoint $\oGf^L$ and a right adjoint $\oGf^R$. The restricted functor
$\oGf: \O_\hlambda \to \oGf(\O_\hlambda)$ is a quotient functor, i.e. the adjunction map $V \mapsto \oGf \oGf^L V$ is an isomorphism for $V \in \oGf(\O_\hlambda)$.
\end{cor}
\begin{proof}
The functors $\oGf^L$ and $\oGf^R$ are obtained by transporting the functors $\tau^L$ and $\tau^R$ from \eqref{tauL} and \eqref{tauR} by means of the equivalence $\oGf \cong \tau$.
Let $V = \oGf(M)$, $M \in \O_\hlambda$. We must prove that the adjunction $V \to \oGf \oGf^L V$ is an isomorphism. 

We first show that $\D_\lambda(f) \mapsto \oGf \oGf^L \D_\lambda(f)$ is an isomorphism. 
For this purpose it is obviously enough to show that the adjunction $\Delta_\lambda(f) \isoto \tau\tau^L(\Delta_\lambda(f))$ is an isomorphism.
Recall that $\tau$ equals the composition
$$\O_{\hlambda} \overset{\kappa}{\to}
\O'_{\lambda} 
\overset{\Theta^\mu_{\lambda}}{\longrightarrow} \O'_{\hmu}
\overset{\sigma}{\to} \N^f_{\hlambda}.$$
The left adjoint of  $\O'_{\lambda} 
\overset{\Theta^\mu_{\lambda}}{\longrightarrow} \O'_{\hmu}$ is $\Theta^\lambda_\mu(  \ ) \otimes_U U_\lambda$, see \eqref{tauL}. Since $\kappa$ and $\sigma$ are equivalences and $\sigma^{-1} \D_\lambda(f) \cong \D_\mu$
it suffices to show that
the adjunction map
$\D_\mu \mapsto \Theta^\mu_\lambda (\Theta^\lambda_\mu (\D_\mu) \otimes_U U_\lambda)$
is an isomorphism. This is proved in Lemma \ref{adjointlem2} below.

\medskip

Note that since $\oGf$ commutes with projective functors and morphisms between them also its left and right adjoints do. Thus, for any projective functor $\P$, the adjunction map $\P(\D_\lambda(f)) \to \oGf \oGf^L(\P(\D_\lambda(f))$ is equivalent
to the map $\P[\D_\lambda(f) \to \oGf \oGf^L(\D_\lambda(f)]$ which is an isomorphism by the above.

Now pick a projective resolution $\P'(\D_\lambda) \to \P(\D_\lambda) \to M \to 0$.  Applying $\oGf$ we get the exact sequence
$$
\P'(\D_\mu) \to \P(\D_\mu) \to V \to 0.
$$
Applying the adjunction morphisms vertically we get a commutative diagram
$$\xymatrix{
\P'(\D_\mu) \ar[d]^{\cong} \ar[r] &\P(\D_\mu)\ar[r] \ar[d]^{\cong}  & V \ar[r] \ar[d] &  0\\
\oGf \oGf^L \P'(\D_\mu)  \ar[r] &\oGf \oGf^L \P(\D_\mu)  \ar[r]& \oGf \oGf^L V \ar[r] & 0}
$$
Therefore $V \to \oGf \oGf^L V$ is an isomorphism.
\end{proof}
Recall that the left adjoint of the functor $\O'_{\lambda} \overset{\Theta^\mu_{\lambda}}{\longrightarrow} \O'_{\hmu}$ is $\Theta^\lambda_\mu(  \ ) \otimes_U U_\lambda$.  
\begin{lemma}\label{adjointlem2} The adjunction morphism $\pi: \D_\mu \mapsto \Theta^\mu_\lambda  (\Theta^\lambda_\mu (\D_\mu) \otimes_U U_\lambda)$ is an isomorphism.
\end{lemma}
\begin{proof}  We shall prove that  $\Theta^\mu_\lambda(\Theta^\lambda_\mu (\D_\mu) \otimes_U U_\lambda) \cong \D_\mu$. From this it will follow that $\pi$ is an isomorphism because $\pi$ is nonzero
(as it corresponds to the natural surjection $\Theta^\lambda_\mu \D_\mu \mapsto  (\Theta^\lambda_\mu \D_\mu) \otimes_U U_\lambda$ under adjunction) and any nonzero endomorphism of a Verma module is an isomorphism.

We have $\Theta^\lambda_\mu (\D_\mu) = P_{w \cdot \lambda}$, where $w$ is the longest element in $W^\mu$. 
Write $P = P_{w \cdot \lambda}$. Let $n_x = (P:\D_{x \cdot \lambda}) = [\D_{x\cdot \lambda}: L_{w \cdot \lambda}]$. We shall show that $P$ is multiplicity free, i.e. that $n_x \leq 1$ for all $x \in W$.
Since $n_{x} \leq n_e$ it suffices to show that $n_e=1$. Since $\D_\lambda$ is projective we have
$$\dim \Hom(\D_\lambda, P) = [P: L_\lambda] = \sum_{x \in W} (P: \D_{x \cdot \lambda})\cdot [\D_{x \cdot \lambda}: L_\lambda] =(P: \D_\lambda) = n_e.$$
But on the other hand
$$\Hom(\D_\lambda, P) = \Hom(\D_\lambda, \Theta^\lambda_\mu \D_\mu) = 
\Hom( \Theta^\mu_\lambda \D_\lambda, \D_\mu) = \Hom(\D_\mu, \D_\mu)$$
and therefore $n_e= \dim \Hom(\D_\mu, \D_\mu) = 1$.

This implies that $P^{J_{\lambda}} \cong \D_\lambda$, see e.g. \cite{Bac01}. Therefore $[P^{J_{\lambda}}:\D_{w_0 \cdot \lambda}] = 1$. 

\medskip

We claim that this implies that $[P/J_{\lambda}P: \D_{w_0 \cdot \lambda}] = 1$. Since we have a surjection $P/J_\lambda P \to \D_{w \cdot \lambda}$ it is enough to show that $[P/J_{\lambda}P: \D_{w_0 \cdot \lambda}] \leq 1$.
Let $I = P^* = \Theta^\lambda_\mu (\nabla_\mu) $ (= injective hull of $L_{w \cdot \lambda}$) and
$T =   \Theta^\lambda_\mu (\D_{w_0 \cdot \mu})$ (a tilting module). Since $(P/J_{\lambda}P)^*  =I^{J_\lambda}$ it is enough to show that $[I^{J_\lambda}: \D_{w_0 \cdot \lambda}] \leq 1$.

Now $T \subseteq P$ and hence $T^{J_\lambda} \subseteq P^{J_\lambda}$ so that $[T^{J_\lambda}:\D_{w_0 \cdot \mu}] \leq 1$.
Let $K = \Ker(\nabla_{\mu} \twoheadrightarrow  \D_{w_0 \cdot \mu})$. 
Then we have $[K : \D_{w_0 \cdot \mu}] =  0$ and therefore  $[\Theta^\lambda_\mu K : \D_{w_0 \cdot \lambda}] = 0$.
The exact sequence $0 \to  \Theta^\lambda_\mu K  \to I \to T \to 0$ gives the sequence
$$
0 \to (\Theta^\lambda_\mu K)^{J_\lambda} \to I^{J_\lambda} \to T^{J_\lambda}.
$$ 
Thus,  $[I^{J_\lambda}:\D_{w_0 \cdot \mu}] \leq [T^{J_\lambda}:\D_{w_0 \cdot \mu}] \leq 1$. This proves the claim.

\medskip

Consider now a Verma flag
$$
0 \subset F_1 \subset F_2 \subset \ldots \subset F_N = P, \ F_i/F_{i-1} \cong \D_{x_i \cdot \lambda}, 
$$
so that $N = [P: \D_{w_0 \cdot \lambda}]$. Then we have $x_i \cdot \lambda \geq w \cdot \lambda$ for each $i$.  Since $[P/J_\lambda P: \D_{w_0 \cdot \lambda}] = 1$ we have
 $[J_\lambda P: \D_{w_0 \cdot \lambda}] = N-1$. Clearly $J_\lambda F_i \subseteq F_{i-1}$ and the latter equality implies that the image of $J_\lambda F_{i}$ in $F_{i-1}/F_{i-2} \cong \D_{x_i}$ 
contains the copy of $\D_{w_0 \cdot \lambda}$ for all $i \geq 2$.  Let $t_i \in F_i$ be such that its image in $F_i/F_{i-1}$ is a highest weight vector of weight $w \cdot \lambda$. Pick $s_i \in U(\n_-) t_i$ such that the image of $s_i$ in $F_i/F_{i-1}$ is a highest weight vector of weight $w_0 \cdot \lambda$. Since the image of $J_\lambda \cdot s_i \neq 0$ in $F_{i-1}/F_{i-2}$ it follows that also the image of  $J_\lambda \cdot t_i \neq 0$ in $F_{i-1}/F_{i-2}$. This shows that $[J_\lambda P: L_{w \cdot \lambda}] \geq N-1$. 
This implies that $[\Theta^\mu_\lambda J_\lambda P: L_{\mu}] \geq N-1$.

On the other hand we have a short exact sequence
$$0 \to \Theta^\mu_\lambda  J_\lambda P \to  \Theta^\mu_\lambda P \to \Theta^\mu_\lambda(P/J_\lambda P) \to 0.$$
It follows that  $\Theta^\mu_\lambda  J_\lambda P$ contains a submodule of $\Theta^\mu_\lambda  P = \D^N_{\mu}$ isomorphic to $\D^{N-1}_{\mu}$. Since moreover $\Theta^\mu_\lambda  (P/J_\lambda P) \twoheadrightarrow \D_\mu$ we conclude that  $\Theta^\mu_\lambda  J_\lambda P \cong \D^{N-1}_{\mu}$ and that $\Theta^\mu_\lambda  (P/J_\lambda P) \cong \D_\mu$.
\end{proof}
\begin{remark} Assume that $f$ is non-degenerate so that by Proposition \ref{aprop}  $Wh_f \circ \oGf \cong \V$. Hence the left adjoint of $\oGf$ is $P_{w_0\cdot \lambda} \otimes_{Z} \overline{Wh}_f$ in this case.
\end{remark}
\begin{remark} If $M \in \N^f_\hlambda$ (for arbitrary $f$) is induced from a Whittaker module over the reductive Lie algebra $\l$ (see Section \ref{xxxyyy}), 
then the adjoints $\oGf^L(M)$ and $\oGf^R(M)$ can be calculated directly only using the definition of $\oGf$ without using
$\tau^L$ and $\tau^R$.
\end{remark}
\subsection{Calculation of $\oGf^L$ and $\oGf^R$.}\label{lastsection}
We end this paper by calculating $\oGf^L$ and $\oGf^R$ on some standard Whittaker modules. Note that $\oGf^L: \N^f_{\hlambda} \to \O_\hlambda$ is right exact and
$\oGf^R: \N^f_{\hlambda} \to \O_\hlambda$ is left exact. Recall that they commute with all projective functors.

We have
$$\oGf^R(\D_{x \cdot \lambda}(f)) \cong \tau^R(\D_{x \cdot \lambda}(f))  \cong \kappa^{-1}\circ ( \Theta^\lambda_\mu( \  ))^{J_\lambda} \circ \sigma^{-1}(\D_{x \cdot \lambda}(f)) \cong \kappa^{-1}\circ \Theta^\lambda_\mu( \Delta_{x \cdot \mu})^{J_\lambda}.$$
If $x = e$ then we have $\Theta^\lambda_\mu( \Delta_{\mu})^{J_\lambda} \cong \Delta_\lambda$, as explained in the proof of Lemma \ref{adjointlem2}, and therefore $\oGf^R(\D_{\lambda}(f)) \cong \D_\lambda$. When
$x = w_0$ one can prove that  $T = \Theta^\lambda_\mu( \Delta_{w_0 \cdot \mu})$ is a multiplicity free tilting module of highest weight $w_0w \cdot \lambda$ and that in this case  $T^{J_\lambda} = \Delta_{w_0w \cdot \lambda}$, where $w$ is the longest element in $W_\mu$. Hence, $\oGf^R(\D_{w_0 \cdot \lambda}(f)) = \Delta_{w^{-1}w_0 \cdot \lambda}$. 

Similarly, we have
$$\oGf^L(\D_{x \cdot \lambda}(f)) \cong  \kappa^{-1}\circ ( \Theta^\lambda_\mu( \  ) \otimes_U U_\lambda) 
\circ \sigma^{-1}(\D_{x \cdot \lambda}(f)) \cong \kappa^{-1}\circ (\Theta^\lambda_\mu( \Delta_{x \cdot \mu})\otimes_U U_\lambda).$$
Now for $x=e$ we have  $\Theta^\lambda_\mu( \Delta_{\mu})\otimes_U U_\lambda = P_{w \cdot \lambda}/J_\lambda P_{w \cdot \lambda}$ so that 
$\oGf^L(\D_{\lambda}(f)) \cong \kappa^{-1} P_{w \cdot \lambda}/J_\lambda P_{w \cdot \lambda}$.
On the other hand, $T/J_\lambda T =  \nabla_{w_0w \cdot \lambda}$ (as follows from the above and selfduality of $T$) and therefore
$\oGf^L(\D_{w_0 \cdot \lambda}(f)) \cong  \nabla_{ww_0 \cdot \lambda}$.

 \bibliographystyle{plain}
  \bibliography{BibWhit}

\end{document}